\documentclass[11pt]{amsart}
\usepackage{amsmath,amsthm,amssymb,tikz}
\usepackage[applemac]{inputenc}
\usepackage{amsfonts,txfonts}

\theoremstyle{theorem}
\newtheorem{theorem}{Theorem}
\newtheorem{lemma}{Lemma}
\newtheorem{prop}{Proposition}
\theoremstyle{remark}

\newtheorem*{remark*}{Remark}

\numberwithin{equation}{section}

\renewcommand{\d}{\mathop{}\mathopen{}\mathrm{d}}

\newcommand{\R}{\mathbb{R}}

\newcommand\gd{\mathrm{geo}d}

\begin{document}
\author{Antonin Monteil, Filippo Santambrogio}
\title{Metric methods for heteroclinic connections}
\address{Laboratoire de Math\'ematiques d'Orsay, Univ. Paris-Sud, CNRS, Universit\'e Paris-Saclay, 91405 Orsay Cedex, France,\\ {\tt antonin.monteil@math.u-psud.fr, filippo.santambrogio@math.u-psud.fr} }
\maketitle

\begin{abstract}We consider the problem $\min\int_\R \frac{1}{2}|\dot{\gamma}|^2+W(\gamma)\d t $ among curves connecting two given wells of $W\geq 0$ and we reduce it, following a standard method, to a geodesic problem of the form $\min\int_0^1 K(\gamma)|\dot{\gamma}|\d t$ with $K=\sqrt{2W}$. We then prove existence of curves minimizing this new action just by proving that the distance induced by $K$ is proper (i.e. its closed balls are compact). The assumptions on $W$ are minimal, and the method seems robust enough to be applied in the future to some PDE problems.
    \end{abstract}

\section{Introduction}

The minimization of an energy such as
\begin{equation}\label{het1}
\left(\gamma:I\to\R^d\right)\mapsto \int_I\left( \frac{1}{2}|\dot{\gamma}|^2(t)+W(\gamma(t))\right)\d t 
\end{equation}
is a very common problem in many mathematical issues, first of all because of its meaning in classical mechanics (where it corresponds to kinetic  + potential energy). The corresponding Euler-Lagrange equation $\gamma''=\nabla W(\gamma)$ represents the simplest example of motion according to the Newton's law where the force producing the acceleration is of gravitational type. 
The same minimization problem and the same ODE also appear in other issues, for instance in phase transition models, where a suitable rescaling of the curve $\gamma$ gives the optimal transition between two states (we refer for instance to \cite{Braides} for a general introduction to this field). For many applications, the case where $I=\R$, $W\geq 0$ and $\gamma$ connects two wells of $W$ (i.e. $\gamma(\pm\infty)=x^\pm$ with $W(x^\pm)=0$) is the most interesting one. The optimal curve $\gamma$ is called a heteroclinic connection (in contrast with the homoclinic connections, which are solutions of $\gamma''=\nabla W(\gamma)$ but with same limits at $\pm \infty$).

The existence of a heteroclinic connection is a delicate problem, because of the lack of compactness of the set $H^1(\R)$ and of the invariance by translations of the action to be minimized. Many ways to overcome this problem have been proposed, under suitable assumptions on $W$ (on its degeneracy or radial monotonicity near the wells, for instance). We cite \cite{SteRocky} as a first analysis of this problem, and many more recent papers, in particular \cite{Ali Fus,AntSmy,Sourdis}. This last paper, \cite{Sourdis}, is the one with the most general result, as it removes the monotonicty assumptions of \cite{Ali Fus} around the wells. In \cite{Sourdis} there is the assumption $\liminf_{|x|\to\infty}W(x)>0$, but it is easy to see that it can weakened into something like $\sqrt{W(x)}\geq k(|x|)$ with $\int_0^\infty k(t)dt=+\infty$, as we do in this paper. Note that \cite{Ali Fus} already used a similar assumption, in the form $\liminf_{|x|\to\infty}|x|^2W(x)=+\infty$, but ours is weaker, and optimal (it is easy to build example of cases where the minimum is not attained without it). 

The idea behind the method that we propose here, very much different from \cite{Ali Fus,Sourdis}, is classical: reduce the problem to a geodesic problem for a weighted metric with a cost given by $K(x):=\sqrt{2 W(x)}$, i.e., instead of minimizing \eqref{het1}, solving
$$
\min \int_0^1\left( \sqrt{2W(\gamma(t))}\, |\gamma'(t)|\right)\d t 
$$
with given initial and final data. The difficulty in this problem is the fact that $K$ is not bounded from below, which makes it difficult to obtain bounds on a minimizing sequence. Instead, we propose an abstract metric approach: we show that the distance $d_K$ induced by the weight $K$ makes $\R^d$ a proper space, which automatically means that it admits the existence of geodesics.

We present our approach in the framework of a general metric space $X$ instead of $\R^d$ in order to prepare possible later extensions to higher dimensional problems, i.e. attacking
$$
\min \int_{\R\times I}\left( \frac{1}{2}|\nabla u|^2(x)+W(u(x))\right)\d x 
$$
where $x=(x_1,x_2)$, and boundary data are fixed as $x_1\to\pm\infty$. This can be interpreted in our framework using $x_1$ as $t$ and $X$ to be $L^2(I)$, with an effective potential of the form $u\mapsto \int_I \frac{1}{2}|\partial_{x_2} u|(x_2)^2+W(u(x_2))\d x_2$. But this obviously raises extra difficulties due to the lack of compactness in infinite dimensions.

The paper is organised as follows: first we recall the main notions concerning curves and geodesics in metric spaces, then we consider the problem of minimizing a weighted length in a metric space, with a weight $K$ which can possibly vanish, then we apply this result to the problem of heteroclinic connections.


\section{Minimal length problem in metric spaces}

Let $(X,d)$ be a metric space, a standard situation being $X=\R^d$ endowed with the Euclidean distance.

\paragraph{\bf Curve in $(X,d)$} A \emph{curve} is a continuous map $\gamma:I\to X$, where $I\subset \R$ is a non-empty interval. 
We denote the set of Lipschitz maps (resp. locally Lipschitz maps) from $I$ to $X$ by $\mathcal{L}(I,X)$ (resp. $\mathcal{L}_{loc}(I,X)$). We also need to introduce the set of \emph{piecewise locally Lipschitz maps}:
\[
\mathcal{L}_{ploc}(I,X):=\big\{\gamma\in \mathcal{C}(I,X)\;:\;\exists t_0=\inf I<t_1<\dots <t_n=\sup I,\, \forall i,\,\gamma\in\mathcal{L}_{loc}(I\cap (t_i,t_{i+1}))\big\}.
\]

\paragraph{\bf Length of a curve}
Given any curve $\gamma:I\to X$, we define the length of $\gamma$ by the usual formula
$$L_d(\gamma):=\sup \sum_{i=0}^{N-1} d(\gamma(t_i),\gamma(t_{i+1}))\in\R\cup\{+\infty\},$$
where the supremum is taken over all $N\geq 1$ and all sequences $t_0\leq \dots\leq t_N$ in $I$. A curve $\gamma$ is said to be \emph{rectifiable} if $L(\gamma)<\infty$.

\paragraph{\bf Length of locally Lipschitz curves} For piecewise locally Lipschitz maps we have the following representation formula for the length: 
\begin{prop}\label{length_lip}
Given $\gamma\in\mathcal{L}_{ploc}(I,X)$, the following quantity,
\[
|\dot{\gamma}|(t)=\lim\limits_{s\to t}\frac{d(\gamma(t),\gamma(s))}{|t-s|},
\]
is well defined for a.e. $t\in I$ and measurable. $|\dot{\gamma}|$ is called \emph{metric derivative} of $\gamma$. Moreover, one has
\[
L_d(\gamma)=\int_I |\dot{\gamma}|(t)\d t .
\]
\end{prop}
We refer for instance to \cite{AmbTil} for the notion of metric derivative and for many other notions on the analysis of metric spaces.

\paragraph{\bf Parametrization}
If $\gamma:I\to X$ is a curve, and $\varphi:I'\to I$ is a non-decreasing surjective continuous mapping, called \emph{parametrization}, then the curve $\sigma=\gamma\circ\varphi :I'\to X$ satisfies $L_d(\sigma)=L_d(\gamma)$. The curve $\gamma$ is said to have \emph{constant speed} if for all $t,t'\in I$ s.t. $t<t'$, $L_d(\gamma_{|(t,t')})=\lambda |t-t'|$. $\lambda$ is the \emph{speed} of the curve $\gamma$. Note that $\gamma$ has constant speed $\lambda$ if and only if $\gamma$ is Lipschitz and $|\dot{\gamma}(t)|=\lambda$ a.e. 
The curve $\gamma$ is {parametrized by arc length} if $\lambda=1$. Assume that a curve $\gamma$ satisfies $L_d(\gamma_{|J})<\infty$ for all compact subset $J\subset I$: then, it is well-known that there exists a reparamatrization of $\gamma$ parametrized by arc length. 
Up to renormalization, it is always possible to consider curves defined on $I=[0,1]$.

\paragraph{\bf Minimal length problem} We define the intrinsic pseudo-metric $\gd$ (called {\it geodesic distance}) by minimizing the length of all curves $\gamma$ connecting two points $x^\pm\in X$:
\begin{equation}
\label{intrinsic}\gd(x^-,x^+):=\inf\{L_d(\gamma)\;:\; \gamma:x^-\mapsto x^+\}\in [0,+\infty],
\end{equation}
where the notation $\gamma:x^-\mapsto x^+$ means that $\gamma$ is a \emph{path} from $x^-$ to $x^+$: there exists $a^-\leq a^+$ s.t. $\gamma\in\mathcal{C}^0([a^-,a^+],X)$ with $\gamma(a^\pm)=x^\pm$. Here, if $a^+$ or $a^-$ is infinite, we use the convention $\gamma(\pm\infty):=\lim_{t\to \pm\infty}\gamma (t)=x^\pm$, if the limit exists. 

When $(X,d)$ is a Euclidean space, $\gd=d$ and the infimum value in \eqref{intrinsic} is achieved by the segment $[a^-,a^+]$. In general, a metric space such that $\gd=d$ is called $\emph{length space}$.

The minimal length problem consists in finding a curve $\gamma:x^-\mapsto x^+$ such that $L_d(\gamma)=\gd(x^-,x^+)$. The existence of such a curve, called \emph{minimizing geodesic}, is given by the classical theorem (see \cite{AmbTil}, for instance):
\begin{theorem}\label{exist_geod}
Assume that $(X,d)$ is proper, i.e. every bounded closed subset of $(X,d)$ is compact. Then, for any two points $x^\pm$ such that $\gd(x^+,x^-)<+\infty$, there exists a minimizing geodesic joining $x^-$ and $x^+$.
\end{theorem}

\section{Minimal length problem in weighted metric spaces}

Let $(X,d)$ be a metric space and $K: X\to \R^+$ be a nonnegative function called \emph{weight function}. From now on, we make the following assumptions on $(X,d,K)$:
\begin{description}
\item[(H1)]
$(X,d)$ is a proper length metric space. 
\item[(H2)]
$K$ is continuous and $\Sigma:=\{K=0\}$ is finite.
\item[(H3)]
For all $x\in X$, $K(x)\geq k(d(x,\Sigma))$ for some function $k\in C^0(\R^+,\R^+)$ with $\int_0^\infty k(t)\d t=+\infty$.
\end{description}
Assumption {\bf (H1)} is satisfied in particular by any Euclidean space. The confining property {\bf (H3)} is fulfilled whenever $\liminf_{d(x_0,x)\to\infty} K(x)>0$ for instance.

Our aim is to investigate the existence of a curve $\gamma \in\mathcal{L}_{ploc}(I,X)$ minimizing the $K$-length, defined by
\[
L_K(\gamma):=\int_I K(\gamma(t))\, |\dot{\gamma}(t)|\d t .
\]
Namely, we want to find a curve $\gamma \in\mathcal{L}_{ploc}(I,X)$ which minimizes the $K$-length between given points $x^\pm\in X$:
\[
d_K(x^-,x^+):=\inf\{L_K(\gamma)\;:\; \gamma\in\mathcal{L}_{ploc}(I,X)\text{ s.t. }\gamma:x^-\mapsto x^+\}.
\]
We are going to prove that $d_K$ is a metric on $X$ s.t. $(X,d_K)$ is proper and $L_K=L_{d_K}$, thus implying the existence of a geodesic between two joinable points, in view of Theorem \ref{exist_geod} (see Theorem \ref{exist_Kgeod} below).
\begin{prop}\label{topo_dK}
The quantity $d_K$ is a metric on $X$. Moreover $(X,d_K)$ enjoys the following properties
\begin{enumerate}
\item
$d_K$ and $d$ are equivalent ( i.e. they induce the same topology) on all $d$-compact subsets of $X$.
\item
$(X,d_K)$ is a proper metric space.
\item
Any locally Lipschitz curve $\gamma:I\to X$ is also $d_K$-locally Lipschitz and the metric derivative of $\gamma$ in $(X,d_K)$, denoted by $|\dot{\gamma}|_K$, is given by $|\dot{\gamma}|_K(t)=K(\gamma(t))\, |\dot{\gamma}|(t)$ a.e.
\item
We have $L_K(\gamma)=L_{d_K}(\gamma)$ for all $\gamma\in\mathcal{L}_{ploc}(I,X)$. 
\end{enumerate}
\end{prop}
\begin{theorem}\label{exist_Kgeod}
For any $x,y\in X$, there exists $\gamma\in \mathcal{L}_{ploc}(I,X)$ s.t. $L_K(\gamma)=d_K(x,y)$ and $\gamma:x\mapsto y$.
\end{theorem}
\begin{proof}
Let us see how Proposition \ref{topo_dK} implies Theorem \ref{exist_Kgeod}. As $(X,d_K)$ is a proper metric space, Theorem \ref{exist_geod} insures the existence of a $L_{d_K}$-minimizing curve $\gamma:x\mapsto y$. Up to renormalization, one can assume that $\gamma$ is parametrized by $L_{d_K}$-arc length. By minimality, we also know that $\gamma$ is injective and thus, $\gamma$ meets the finite set $\{K=0\}$ at finite many instants $t_1<\dots < t_N$. As $K$ is bounded from below by some positive constant on each compact subinterval of $(t_i,t_{i+1})$ for $i\in\{1,\dots,N\}$, Lemma \ref{est_dK} below implies that $\gamma$ is piecewise locally $d$-Lipschitz. Finally, thanks to Statement 4 of Proposition \ref{topo_dK}, the fact that $\gamma$ minimizes $L_{d_K}$ means that it also minimizes $L_K$ among $\mathcal{L}_{ploc}$ curves connecting $x$ to $y$.
\end{proof}
In order to prove Proposition \ref{topo_dK}, we will need the following estimations on $d_K$.
\begin{lemma}\label{est_dK}
For all $x,y\in X$, one has
\[
K_{d(x,y)}(x)\, d(x,y)\leq d_K(x,y)\leq K^{d(x,y)}(x)\, d(x,y),
\]
where $K_r(x)$ and $K^r(x)$ are defined for any $r\geq 0$ and $x\in X$ by 
\[
K_r(x):=\inf\{K(y)\;:\; d(x,y)\leq r\},\quad K^r(x):=\sup\{K(y)\;:\; d(x,y)\leq r\}.
\]
\end{lemma}
\begin{proof}
Set $r:=d(x,y)$. Since any curve $\gamma:x\mapsto y$ has to get out of the open ball $B:=B_d(x,r)$, it is clear that 
\[
L_K(\gamma)=\int_I K(\gamma(t))\, |\dot{\gamma}|(t)\d t \geq r \inf_B K=rK_r(x).
\] Taking the infimum over the set of curves $\gamma\in\mathcal{L}_{ploc}$ joining $x$ and $y$, one gets the first inequality. 

For the second inequality, let us fix $\varepsilon>0$. By construction, there exists a Lipschitz curve $\gamma:x\mapsto y$, that one can assume to be parametrized by arc-length, s.t. $L_d(\gamma)\leq r+\varepsilon$. In particular, $\operatorname{Im}(\gamma)$ is included in the ball $B_d(x,r+\varepsilon)$. Thus, one has 
\[
d_K(x,y)\leq L_K(\gamma)\leq (r+\varepsilon)\, K^{r+\varepsilon}(x)
\]
and the second inequality follows by sending $\varepsilon\to 0$. Indeed, the mapping $r\to K^r(x)$ is continuous on $[0,+\infty)$ since $K$ uniformly continuous on compact sets and since bounded closed subsets of $X$ are compact (assumption {\bf (H1)}).
\end{proof}

%

\begin{proof}[Proof of Proposition \ref{topo_dK}]
The proof is divided into six steps.

\textsc{Step 1: $d_K$ is a metric.}

First note that $d_K$ is finite on $X\times X$. Indeed, given  two points $x,y\in X$, just take a Lipschitz curve connecting them, and use $L_K(\gamma)\leq L_d(\gamma)\sup_{\operatorname{Im}(\gamma)} K<+\infty$. The triangle inequality for $d_K$ is a consequence of the stability of the set $\mathcal{L}_{ploc}$ by concatenation. The fact that $d_K(x,y)=0$ implies $x=y$ is an easy consequence of the finiteness of the set $\{K=0\}$. Indeed, if $x\neq y$, then any curve $\gamma:x\mapsto y$ has to connect $B_d(x,\varepsilon )$ to ${B_d}^c(x,2\varepsilon)$ for all $\varepsilon>0$ small enough. This implies that $L_K(\gamma)\geq \varepsilon\, \inf_C K$, where $C=\{y\;:\; \varepsilon\leq d(x,y)\leq 2\varepsilon\}$. But for $\varepsilon$ small enough, $C$ does not intersect the set $\{K=0\}$ so that $\inf_C K>0$. In particular, $d_K(x,y)\geq \varepsilon\, \inf_C K>0$.

\textsc{Step 2: $d_K$ and $d$ are equivalent on $d$-compact sets.}

Take $Y\subset X$ a compact set, and suppose $Y\subset B_d(x_0,R)$ just to fix the ideas. Consider the identity map from $(Y,d)$ to $(Y,d_K)$. It is an injective map between metric spaces. Moreover, it is continuous, since, as a consequence of Lemma \ref{est_dK}, we have $d_K\leq Cd$ on $Y\times Y$, where $C=\sup_{B_d(x_0,3R)} K<+\infty$ (note that the closed ball $B_d(x_0,3R)$ is $d$-compact, and that we supposed $d=\gd$ since $(X,d)$ is a length space). Hence, as every injective continuous map defined on a compact space is a homeomorphism, $d$ and $d_K$ are equivalent (on $Y$). 

\textsc{Step 3: every closed ball in $(X,d_K)$ is $d$-bounded}

This is a consequence of assumptions ${\bf (H1)}$ and ${\bf (H3)}$. Let us take $x_0,x\in X$ with $d_K(x,x_0)\leq R$. By definition, there exists $\gamma\in  \mathcal{L}_{ploc}(I,X)$ s.t. $\gamma :x_0\mapsto x$ and $L_K(\gamma)\leq d_K(x_0,x)+1$. Now, set $\phi(t):=d(\gamma(t),\Sigma)$: since the function $x\mapsto d(x,\Sigma)$ has Lipschitz constant equal to $1$, we have $\phi\in  \mathcal{L}_{ploc}(I,\R)$ and $|\phi'(t)|\leq |\gamma'(t)|$ a.e. Take $h:\R^+\to\R^+$ the antiderivative of $k$, i.e. $h'=k$ with $h(0)=0$, and compute $[h(\phi(t))]'=k(\phi(t))\phi'(t)$. Hence,
$$|[h(\phi(t))]'|=k(\phi(t))|\phi'(t)|\leq K(\gamma(t))\, |\gamma'(t)|$$
and $h(d(\gamma(t),\Sigma))\leq h(d(x_0,\Sigma))+L_K(\gamma)\leq h(d(x_0,\Sigma))+R+1.$ Since $\lim_{s\to\infty}h(s)=+\infty$, this provides a bound on $d(x,\Sigma)$ which means that the ball $B_{d_K}(x_0,R)$ is $d$-bounded.


\textsc{Step 4: every closed ball in $(X,d_K)$ is $d_K$-compact}

Now that we know that closed ball in $(X,d_K)$ are $d$-bounded, since $(X,d)$ is proper, we know that they are contained in $d$-compact sets. But on this sets $d$ and $d_K$ are equivalent, hence these balls are also $d$-closed, hence $d$-compact, and thus $d_K$-compact.


\textsc{Step 5: proof of statement 3.}
Let $\gamma:I\mapsto X$ be a $d$-locally Lipschitz curve valued in $X$. Thanks to the second inequality in Lemma \ref{est_dK}, $\gamma$ is also $d_K$-locally Lipschitz. Now, Lemma \ref{est_dK} provides
\[
K_{r}(\gamma(t))\ \frac{d(\gamma(t),\gamma(s))}{|t-s|}\leq \frac{d_K(\gamma(t),\gamma(s))}{|t-s|}\leq K^{r}(\gamma(t))\ \frac{d(\gamma(t),\gamma(s))}{|t-s|}
\]
with $r:=d(\gamma(t),\gamma(s))$.
In the limit $s\to t$ we get
\[
K(\gamma(t))\, |\dot{\gamma}|(t)\leq |\dot{\gamma}|_K(t)\leq K(\gamma(t))\, |\dot{\gamma}|(t)\ \text{ a.e.},
\]
where the continuity of $r\to K^r(x)$ and $r\to K_r(x)$ on $[0,+\infty)$ has been used.

\textsc{Last step: proof of statement 4.} This is an easy consequence of Statement 3. Indeed, by additivity of $L_K$ and $L_{d_K}$ and since $L_K(\gamma)=\sup L_K(\gamma_J)$,  $L_{d_K}(\gamma)=\sup L_{d_K}(\gamma_J)$, both supremum being taken on compact subsets $J\subset I$, it is enough to prove that $L_K(\gamma)=L_{d_K}(\gamma)$ when $\gamma\in \mathcal{L}(I,X)$. But any curve $\gamma\in \mathcal{L}(I,X)$ is locally $d_K$-Lipschitz and satisfies
\[
L_{d_K}(\gamma)=\int_I |\dot{\gamma}|_K(t)\d t=\int_I K(\gamma(t))\, |\dot{\gamma}|(t)\d t=L_K(\gamma).\qedhere
\]
\end{proof}

\section{Existence of heteroclinic connections}
\noindent Our aim is to investigate the existence of a global minimizer of the energy
\[
E_W(\gamma)=\int_\R \left(\frac{1}{2}|\dot{\gamma}|^2(t)+W(\gamma(t))\right)\d t ,
\]
defined over locally Lipschitz curves $\gamma: x^-\mapsto x^+$ valued in a metric space $(X,d)$. Here $W: X\mapsto\R^+$ is a continuous function, called \emph{potential} in all the sequel, and $x^\pm\in X$ are two wells, i.e. $W(x^\pm)=0$. Note that $W(x^\pm)=0$ is a necessary condition for the energy of $\gamma$ to be finite. The main result of this section is the following:
\begin{theorem}\label{heteroclinic}
Let $(X,d)$ be a metric space, $W:X\mapsto\R^+$ a continuous function and $x^-,x^+$ points of $X$ such that:
\begin{description}
\item[(H)]
$(X,d,K)$ satisfies hypotheses ${\bf H1-3}$ of the previous section, where $K:=\sqrt{2W}$.
\item[(STI)]
$W(x^-)=W(x^+)=0$ and $d_K$ (defined above) satisfies the following strict triangular inequality on the set $\{W=0\}$:  for all $x\in X\setminus\{x^-,x^+\}\,\text{ s.t. }W(x)=0$, $d_K(x^-,x^+)<d_K(x^-,x)+d_K(x,x^+)$.
\end{description}
Then, there exists a heteroclinic connection between $x^-$ and $x^+$, i.e. $\gamma\in \mathcal{L}(\R,X)$ such that
\[
E_W(\gamma)=\inf\{E_W(\sigma)\;:\; \sigma\in\mathcal{L}_{ploc}(\R,X),\, \sigma:x^-\mapsto x^+\}.
\]
Moreover, $E_W(\gamma)=d_K(x^-,x^+)$.
\end{theorem}
\begin{proof}
This theorem is a consequence of Theorem \ref{exist_Kgeod} and the following consequence of Young's inequality:
\begin{equation}\label{young}
\text{for all } \gamma\in\mathcal{L}_{ploc}(\R,X),\ E_W(\gamma)\geq L_K(\gamma),
\end{equation}
where $K:=\sqrt{2W}$. Indeed, thanks to assumption {\bf (H)}, Theorem \ref{exist_Kgeod} provides a $L_K$-minimizing curve $\gamma_0 : I\to X$, that one can assume to be injective and parametrized by $L_K$-arc length, connecting $x^-$ to $x^+$. Thanks to assumption ${\bf (STI)}$, it is clear that the curve $\gamma_0$ cannot meet the set $\{W=0\}$ at a third point $x\neq x^\pm$: in other words $K(\gamma(t))>0$ on the interior of $I$. Thus, 
$\gamma_0$ is also $d$-locally Lipschitz on $I$ (and not only piecewise locally Lipschitz). In particular, one can reparametrize the curve $\gamma_0$ by $L_d$-arc length, so that $|\dot{\gamma_0}|=1$ a.e.

Then, in view of \eqref{young}, it is enough to prove that $\gamma_0$ can be reparametrized in a curve $\gamma$ satisfying $|\dot{\gamma}|=K\circ\gamma$ a.e., so that \eqref{young} becomes an equality. By the way, this automatically implies that $\gamma$ is Lipschitz, since it provides a bound on $|\gamma'|$. Namely, we look for an admissible curve $\gamma:\R\to X$ of the form $\gamma(t)=\gamma_0(\varphi(t))$, where $\varphi:\R\to I$ is $\mathcal{C}^1$, increasing and surjective. For $\gamma$ to satisfy the equipartition condition, i.e. $|\dot{\gamma}|(t)=K(\gamma(t))$ a.e., we need $\varphi$ to solve the ODE
\begin{equation}\label{ode_geod}
\varphi'(t)=F(\varphi(t)),
\end{equation}
where $F:\overline{I}\to\R$ is the continuous function defined by $F=K\circ\gamma_0$ on $I$ and $F\equiv 0$ outside $I $. Thanks to the Peano-Arzel\`a theorem, \eqref{ode_geod} admits at least one maximal solution $\varphi_0:J=(t^-,t^+)\mapsto \R$ such that $0\in J$ and $\varphi_0(0)$ is any point inside $I$. Since $F$ vanishes out of $I$, we know that $\operatorname{Im}(\varphi_0)\subset \overline{I}$. Moreover, since $\varphi_0$ is non decreasing on $I$, it converges to two distinct stationary points of the preceding ODE. As $F>0$ inside $I$, one has $\lim_{t\to t^+} \varphi_0(t)=\sup I$ and $\lim_{t\to t^-} \varphi_0(t)=\inf I$. We deduce that $\varphi_0$ is an entire solution of the preceding ODE, i.e. $I=\R$. Indeed, if $I\neq \R$, say $t^+<+\infty$, then one could extend $\varphi_0$ by setting $\varphi_0(t)=\sup I$ for $t>t^+$. Finally, the curve $\gamma:=\gamma_0\circ\varphi_0$ satisfies $\gamma(\pm\infty)=x^\pm$, $|\dot{\gamma}|(t)=K(\gamma(t))$ a.e. and so
\[
E_W(\gamma)=L_K(\gamma)=L_K(\gamma_0)=d_K(x^-,x^+)\leq\inf\{E_W(\sigma)\;:\; \sigma\in\mathcal{L}_{ploc}(\R,X),\,\gamma:x\mapsto y\}.
\]
Thus, $\gamma$ minimizes $E_W$ over all admissible connections between $x^-$ and $x^+$. 
\end{proof}
\begin{remark*}
\begin{itemize}
\item
It is easy to see that the equirepartition of the energy, that is the identity $|\dot{\gamma}|^2(t)=2W(\gamma(t))$, is a necessary condition for critical points of $E_W$.
\item
The assumption {\bf (STI)} is not optimal but cannot be removed, and is quite standard in the literature. Without this assumption, it could happen that a geodesic $\gamma$ would meet the set $\{W=0\}$ at a third point $x\neq x^\pm$. In this case, it is not possible to parametrize $\gamma$ in such a way that $|\dot{\gamma}|(t)=K(\gamma(t))$.
\item
However, if $K=\sqrt{2W}$ is not Lipschitz, it is possible that there exists a heteroclitic connection $\gamma:x^-\mapsto x^+$ meeting $\{W=0\}$ at a third point $x\neq x^\pm$. Indeed, if $\liminf_{y\to x}K(y)/|y|>0$, then, there exists a heteroclinic connection $\gamma^-:x^-\mapsto x$ which reaches $x$ in finite time (say, $\gamma^-(t)=x$ for $t\geq 0$). Similarly, there exists a heteroclinic connection $\gamma^+:x\mapsto x^+$ such that $\gamma^+(t)=z$ for $t\leq 0$. Thus, there exists a heteroclinic connection between $x^-$ and $x^+$ obtained by matching $\gamma^-$ and $\gamma^+$.
\end{itemize}
\end{remark*}

\end{document}